\documentclass[a4paper,10pt]{amsart}
\usepackage[utf8]{inputenc}
\usepackage{amsxtra}
\usepackage{amsopn}
\usepackage{amsmath,amsthm,amssymb}
\usepackage{amscd}
\usepackage{amsfonts}
\usepackage{latexsym}
\usepackage{verbatim}
\usepackage{color}

\theoremstyle{plain}
\newtheorem{theorem}{Theorem}[section]
\newtheorem*{theorem*}{Theorem}
\newtheorem{definition}[theorem]{Definition}
\newtheorem{lemma}[theorem]{Lemma}
\newtheorem{prop}[theorem]{Proposition}
\newtheorem{cor}[theorem]{Corollary}
\newtheorem{rem}[theorem]{Remark}
\newtheorem{ex}[theorem]{Example}
\newtheorem*{mt*}{Main Theorem}
\makeatletter
\newenvironment{proofof}[1]{\par
  \pushQED{\qed}%
  \normalfont \topsep6\p@\@plus6\p@\relax
  \trivlist
  \item[\hskip\labelsep
        \bfseries
    Proof of #1\@addpunct{.}]\ignorespaces
}{%
  \popQED\endtrivlist\@endpefalse
}
\makeatother

\newcommand\C{{\mathbb C}}

\newcommand\N{{\mathbb N}}
\newcommand\R{{\mathbb R}}
\newcommand\Z{{\mathbb Z}}

\newcommand\Span{{\hbox{\em Span}}}

\newcommand{\del}{\partial}
\newcommand{\delbar}{\overline{\del}}

\title{On $L_2$-cohomology of almost Hermitian manifolds}
\author{Richard Hind and Adriano Tomassini}
\address{Department of Mathematics \\
University of Notre Dame \\
Notre Dame, IN 46556}
\email{hind.1@nd.edu}
\address{Dipartimento di Scienze Matematiche, Fisiche e Informatiche\\
Plesso Matematico e Informatico,
Universit\`{a} di Parma\\
Parco Area delle Scienze 53/A, 43124 \\
Parma, Italy}
\email{adriano.tomassini@unipr.it}

\thanks{The first author is partially supported by Simons Foundation grant \# 317510. \newline The second author is partially supported by the Project PRIN ``Varietà reali e complesse: geometria, topologia e analisi armonica''
and by GNSAGA of INdAM}
\begin{document}

\maketitle

\begin{abstract}
Let $(X,J,\omega,g)$ be a complete $n$-dimensional K\"ahler manifold. A Theorem by Gromov \cite{G} states that the if the K\"ahler form is $d$-bounded, then the
space of harmonic $L_2$ forms of degree $k$ is trivial, unless $k=\frac{n}{2}$. Starting with a contact manifold $(M,\alpha)$ we show that the same conclusion
does not hold in the category of almost K\"ahler manifolds. Let $(X,J,g)$ be a complete almost Hermitian manifold of dimension four. We prove that the reduced $L_2$ $2^{nd}$-cohomology group decomposes as direct sum
of the closure of the invariant and anti-invariant $L_2$-cohomology. This generalizes a decomposition theorem by Dr\v{a}ghici, Li and  Zhang \cite{DLZ} for $4$-dimensional
closed almost complex manifolds to the $L_2$-setting.
\end{abstract}

\section{Introduction}
Cohomological properties of closed complex manifolds have been recently studied by many authors, focusing on their relations with other special structures (see e.g. \cite{LZ,DLZ,ATZ} and the references therein). Starting with an almost complex manifold $(X,J)$,  in \cite{DLZ} $J$-anti-invariant and $J$-invariant cohomology groups are defined; in particular, in \cite{DLZ} it is proved that on a closed almost complex $4$-manifold the $2^{nd}$-de Rahm cohomology group decomposes as the direct sum of $J$-invariant and $J$-anti-invariant cohomology subgroups. \newline
The aim of this paper is to study cohomological properties of non compact almost complex manifolds. In this context, $L_2$-cohomology provides a useful tool to study the relationship between such properties and the existence of further structures, e.g., K\"ahler, almost K\"ahler structures.

In \cite{G} Gromov developed $L_2$-Hodge theory for complete Riemannian manifolds, respectively K\"ahler manifolds, proving an $L_2$-Hodge decomposition Theorem
for $L_2$-forms. \newline As a consequence, for a complete and $d$-bounded K\"ahler manifold $X$, denoting by $\mathcal{H}^k_2$, respectively $\mathcal{H}^{p,q}_2$, the space
of $\Delta$-harmonic $L_2$-forms of degree $k$, respectively $\Delta_{\delbar}$-harmonic $L_2$-forms of bi-degree $(p,q)$, he showed that
$\mathcal{H}^k_2\simeq\oplus_{p+q=k}\mathcal{H}^{p,q}_2$; furthermore, denoting by $m=\dim_\C X$, that $\mathcal{H}^k_2=\{0\}$, for all $k\neq m$ and hence
$\mathcal{H}^{p,q}_2=\{0\}$, for all $(p,q)$ such that $p+q\neq m$. A key ingredient in the proof is the Hard Lefschetz Theorem.

In the present paper we show that such a conclusion no longer holds in the category of non compact almost K\"ahler manifolds. 
Indeed, by using methods of contact
geometry, starting with a contact manifold $(M,\alpha)$ having an exact symplectic filling (see Definition \ref{symplectic-filling}), 
we construct a $d$-bounded complete almost
K\"ahler manifold $Y$ satisfying $L_2H^1(Y)\neq\{0\}$. 
The result is applied to compact quotients of the $(2n-1)$-dimensional Heisenberg group.

Next we focus on $L_2$-cohomology of non compact almost Hermitian $4$-dimensional manifolds. We prove that if $(X,J,g)$ is a 
$4$-dimensional complete almost Hermitian manifold, then the reduced $L_2$-cohomology $L_2H^2(X;\R)$ decomposes as the direct sum 
of the closure of $J$-anti-invariant and $J$-invariant cohomology, namely
$L_2H^2(X;\R)=\overline{L_2H^+(X)}\oplus\overline{L_2H^-(X)}$. This can be viewed as a sort of $L_2$-Hodge decomposition theorem 
for complete almost Hermitian manifolds and it generalizes Dr\v{a}ghici,  Li and  Zhang's Theorem \cite{DLZ} for closed $4$-dimensional 
almost complex manifold to the $L_2$-cohomology.

The paper is organized as follows: in Section 2 we recall some generalities regarding $L_2$-cohomology. 
Section 3 is devoted to the proof of non vanishing of the first $L_2$-cohomology group. In Section 4 we prove the 
decomposition Theorem \ref{decomp} and give cohomological obstructions for an almost complex structure to admit a compatible
complete symplectic form.
\newline
Finally, we would
like mention an open question. An almost-complex manifold of dimension at least $6$ 
may have a taming symplectic form but not a compatible symplectic form (see e.g., \cite{MT}).
For closed $4$-dimensional manifolds however, there are no local obstructions and Donaldson in \cite{D} 
raised the following question:

\medskip
\noindent{\bf Donaldson's Question}(\cite{D}){\em  
If $J$ is an almost complex structure on a compact 4-manifold
which is tamed by a symplectic form, is there a symplectic form compatible
with J?}

Moving to the complex case, it is still unknown whether a closed complex manifold $X$ of dimension at least $6$ with a taming symplectic form also has a 
compatible symplectic form, in other words, whether it is K\"{a}hler. Such a question has a positive answer by Li and Zhang for complex surfaces \cite[Theorem 1.2]{LZ}.
Here is an analogue of the question for open manifolds.

\medskip
\noindent{\bf Question.} {\em Let $(X,J)$ be a complex $2n$-dimensional manifold. Suppose 
there exists a $d$-bounded  symplectic form $\omega$ taming $J$ such that $g(\cdot,\cdot)=\frac{1}{2}(\omega(\cdot,J\cdot)-\omega(J\cdot,\cdot))$ 
is complete. Does $(X,J)$ admit a complete $d$-bounded K\"ahler structure whose corresponding metric is uniformly comparable to $g$?}

Our construction in section 3 gives $d$-bounded complete almost complex
manifolds $Y$ which admit a taming symplectic form and corresponding complete metric satisfying $L_2H^1(Y;\R)\neq\{0\}$.
If our construction could be upgraded to give examples of (integrable) complex manifolds with this property then by Gromov's theorem there could not be a compatible K\"{a}hler structure with comparable metric, thus implying a negative answer.

\medskip
\noindent{\sl Acknowledgments.} We thank Weiyi Zhang for useful comments. The second author would like to thank the Math Departments of Stanford and Notre Dame Universities for their warm hospitality.

\section{Preliminaries}\label{preliminaries}
We start by recalling some notions about $L_2$-cohomology. Let $(X,g)$ be a Riemannian manifold and denote by $\Omega^k(X)$ the space of smooth
$k$-forms on $X$. Then $\alpha\in \Omega^k(X)$ is said to be {\em bounded} if the $L_\infty$-norm of $\alpha$
is finite, namely,
$$
\lVert\alpha\lVert_{L_\infty(X)}=\sup_{x\in X}\vert\alpha(x)\vert\, < +\infty
$$
where $\vert\alpha(x)\vert$ denotes the pointwise norm induced by the metric $g$ on the space of forms.
By definition, a (smooth) $k$-form $\alpha$
is said to be {\em d-bounded} if $\alpha=d\beta$, where $\beta$ is a bounded $(k-1)$-form.
Furthermore,
a $k$-form $\alpha$ is said to be
$L_2$, namely $\alpha\in L_2\Omega^k(X)$ if
$$
\lVert\alpha\lVert_{L_2(X)}:=\left(\int_X\vert\alpha(x)\vert^2\, dx\right)^{\frac{1}{2}}<+\infty,
$$
that is the pointwise norm $\lvert \alpha\lvert^2$ is integrable.
Denote by $(L_2A^\bullet(X),d)$ the sub-complex of $(\Omega^\bullet(X),d)$ formed by differential forms $\alpha$ such that both $\alpha$ and
$d\alpha$ are in $L_2$. Then the {\em reduced $L_2$-cohomology group} of {\em degree} $k$ of $X$ is defined as
$$
L_2H^k(X;\R)=L_2A^k(X)\cap\ker d\Big\slash\overline{dL_2(A^{k-1}(X))}.
$$
We recall the following (see \cite[Lemma 1.1.A]{G})
\begin{lemma}\label{stokes}
Let $(X,g)$ be a complete Riemannian manifold of dimension $n$ and let $\alpha\in L_1\Omega^{n-1}(X)$, that is 
$$
\int_X \vert\alpha(x)\vert\, dx <+\infty.
$$ 
Assume that also $d\alpha\in
L_1\Omega^n(X)$. Then
$$
\int_X d\alpha =0.
$$
\end{lemma}
Let $\Delta = d\delta +\delta d$ denote the Hodge Laplacian and set
$$\mathcal{H}^k_2=\{\alpha\in L_2\Omega^k(X)\,\,\,\vert\,\,\, \Delta\alpha =0\}$$
namely, $\mathcal{H}^k_2$ is the space of harmonic $L_2$-forms on $(X,g)$ of degree $k$. Then, under 
the assumption that $(X,g)$ is complete,
Gromov proved the following Hodge decomposition for $L_2$-forms (see \cite{G}), namely,
\begin{equation}\label{L2Hodge}
L_2\Omega^k(X)= \mathcal{H}^k_2\oplus\overline{d(L_2\Omega^{k-1}(X))}\oplus\overline{\delta(L_2\Omega^{k+1}(X))},
\end{equation}
where $\overline{d(L_2\Omega^{k-1}(X))}$ means the closure in $L_2\Omega^k(X)$ of
$L_2\Omega^k(X)\cap d(L_2\Omega^{k-1}(X))$ and similarly for
$\overline{\delta(L_{2}(A^{k+1}(X))}$.
Given any $\alpha,\beta\in \Omega^k(X)$, we set
$$
\langle\alpha,\beta\rangle=\int_X g(\alpha,\beta)dx.
$$
We have the following
\begin{lemma}\label{d-closed}
Let $(X,g)$ be a complete Riemannian manifold and let $\alpha\in L_2\Omega^k(X)$. Denote by
$$
\alpha =\alpha_H + \lambda+\mu
$$
the Hodge decomposition of $\alpha$, where $\alpha_H\in\mathcal{H}^k_2$, $\lambda\in\overline{d(L_2\Omega^{k-1}(X))}$,
$\mu\in\overline{\delta(L_2\Omega^{k+1}(X))}$. If $d\alpha=0$, then
$\mu=0$.
\end{lemma}
\begin{proof}
Let $\{\delta\mu_j\}_{j\in\N}$ be a sequence of $\delta$-exact forms in $L_2\Omega^k(X)$ such that $\mu_j\in L_2\Omega^{k+1}(X)$ for evey $j\in\N$ and 
$\delta\mu_j\to \mu$ in $L_2$. Then
$$
\langle \alpha,\delta\mu_j\rangle=\langle d\alpha,\mu_j\rangle =0.
$$
Now
$$
\begin{array}{ll}
\vert\langle\alpha,\mu \rangle-\langle \alpha,\delta\mu_j\rangle \vert&=
\vert\int_Xg(\alpha,\mu-\delta\mu_j) dx\vert\leq
\int_X\vert\alpha\vert\,\vert \mu -\delta\mu_j\vert dx\\[10pt]
&\leq \lVert \alpha\lVert_{L_2(X)}\lVert\mu -\delta\mu_j\lVert_{L_2(X)}.
\end{array}
$$
Hence, the numerical sequence $\{\langle\alpha, \delta\mu_j\rangle\}_{j\in\N}$ converges to $\langle \alpha,\mu\rangle$.
and consequenltly, $\langle \alpha,\mu\rangle=0$. Therefore, by the $L_2$-orthogonality of the $L_2$-Hodge decomposition,
it follows that $\mu=0$.
\end{proof}

\section{$L_2$-cohomology and contact structures}
Let now $(X,J)$ be a complex manifold and $g$ be a Hermitian metric. Then according to Gromov \cite[1.2.B]{G}, if $X$ is
a complete $n$-dimensional K\"ahler manifold whose K\"ahler form $\omega$ is $d$-bounded, then $\mathcal{H}^k_2=\{0\}$, unless
$k=\frac{n}{2}$. \newline
In this section we will see that the same conclusion does not hold in
the category of almost K\"ahler manifolds.

To begin, let $M$ be a $(2n-1)$-dimensional compact contact manifold, $n>1$ and denote
by $\alpha$ a contact form. Let $\xi=\ker\alpha$ be the contact distribution and $R$ be the Reeb vector field. \newline
On the product
manifold $X=M \times (3,+\infty)$, with $t$ the coordinate on $(3, \infty)$, let $\rho=\rho(t)$ be a positive smooth function, such that $\rho'>0$
and let $\omega_\rho=d(\rho\alpha)$. Then
$\omega_\rho$ is a symplectic form on $X$.

\begin{definition}\label{symplectic-filling} We say that a contact manifold with contact form $\alpha$ has an {\em exact symplectic filling} if there exists a
compact exact symplectic manifold $(W, \omega = d\lambda)$ with $\partial W = M$ and $\lambda|_M = \alpha$. Furthermore we require the
Liouville field $\zeta$ defined by $\zeta \rfloor \omega = \lambda$ to be outward pointing along $M$.
\end{definition}

We remark that if a particular contact form on $M$ has an exact symplectic filling then so do all other contact forms which generate
the same contact structure, that is, all $\alpha'$ such that $\ker \alpha' = \ker \alpha$.

A version of Darboux' Theorem implies that a tubular neighborhood of $M = \partial W$ in $W$ can be identified symplectically with
$(M \times (-\delta, 0], d(e^t \alpha))$, and we may choose a primitive on $W$ equal to $e^t \alpha$ in this neighborhood.

\begin{prop} Suppose that $(M, \alpha)$ has an exact symplectic filling and $\rho(3)>1$. Then there exists an exact
symplectic manifold $(Y, \omega = d\beta)$
such that the complement of a compact set may be identified with $X = M \times (3,+\infty)$ via a diffeomorphism pulling back $\rho\alpha$ to $\beta$.
\end{prop}

\begin{proof} We set $Y = W \cup (M \times (-\delta, \infty))$ where we identify $M \times (-\delta, 0]$ with a tubular neighborhood
of $M= \partial W$ as above. Then define $\beta|_W = \lambda$ and $\beta|_{M \times (0,\infty)} = \rho(t)\alpha$ where $\rho$ is extended
to $(-\delta,+\infty)$ such that $\rho=e^t$ for $t$
close to $0$ and $\rho' >0$ for all $t>0$.
\end{proof}

\begin{rem}\label{extension}
We note that if $\rho \alpha$ is bounded (with respect to any metric) on $M \times [3,+\infty) \subset Y$ then it is globally bounded; any compatible almost complex
structure on $M \times [3,+\infty)$ extends to a compatible almost complex structure on $Y$; any exact $1$-form $\gamma$ on $M \times [3,+\infty)$
extends to an exact $1$-form on $Y$,
and if $\gamma$ lies in $L_2$ then so does its extension.
\end{rem}
Given a smooth function $f:(3,+\infty)\to (0,+\infty)$ we will always consider almost complex structures $J$ such that
$$
JR=-\frac{1}{f}\frac{d}{dt}\qquad \hbox{\rm and}\qquad J(\xi) = \xi\,.
$$
Suppose that $\eta$ is a $1$-form on $M$ which on a compact set $K \subset M$ is never $0$ and satisfies $\eta(R)=0$. Also assume that
there exists a trivializing frame $\{V_1,\ldots,V_{2n-2}\}$ on the contact
distribution $\xi |_K$ which is a symplectic basis of $d\alpha\vert_\xi$, namely such that
$$d\alpha\vert_\xi=\sum_{i=1}^{n-1}V^{2i-1}\wedge V^{2i},$$
where $\{V^1,\ldots,V^{2n-2},\alpha\}$ denotes the dual coframe of $\{V_1,\ldots,V_{2n-2},R\}$.
Then $\eta |_K =\sum_{i=1}^{2n-2}a_iV^i$, for suitable smooth
functions $a_1,\ldots, a_{2n-2}$ on $K$. \newline
Define $V=\sum_{i=1}^{2n-2}a_iV_i$.
Given a positive smooth function $\varepsilon=\varepsilon(t)$, on $(3,+\infty)$ we say that an {\em almost complex structure
is adapted to $\eta |_K$} if in addition
$$
J_{\varepsilon,f}(V)=\frac{1}{\varepsilon}\sum_{i=1}^{2n-2}a_{2i-1}V_{2i}-a_{2i}V_{2i-1}.
$$
An almost complex structure adapted to $\eta |_K$ is specified globally on the distribution $\langle R,\frac{d}{dt}\rangle$ and on
$K$ on the larger distribution $\langle R,\frac{d}{dt},V, JV\rangle$. Such an almost complex structure can be extended to an almost
complex structure on (the tangent bundle of) $X$ which
is compatible with $\omega_\rho$ and such that $J_{\varepsilon,f}(\xi)\subset\xi$. \newline
We will denote by $g_{\varepsilon,f,\rho}$ the Riemannian metric
associated with $(\omega_\rho,J_{\varepsilon,f})$, that is $g_{\varepsilon,f,\rho}(\cdot,\cdot)=\omega_\rho (\cdot,J_{\varepsilon,f}\cdot)$.

\begin{theorem}\label{almost-kaehler}
Let $\rho(t)=\log t$, $\varepsilon(t)=\rho(t)t^{1-n}$ and $f(t)=\frac{1}{t\log^2t}$. Then $(X,\omega_\rho,J_{\varepsilon,f})$
is an almost K\"ahler manifold, and if $(M, \alpha)$ has an exact symplectic filling then the structure extends to $Y$.
Further:
\begin{itemize}
\item[i)] $\omega_\rho$ is $d$-bounded.
\item[ii)] $(Y,\omega_\rho,J_{\varepsilon,f},g_{\varepsilon,f,\rho})$ is complete.
\item[iii)] Suppose that $\gamma$ is a $1$-form on $M$ such that there exists a trivializing frame on
the support of $\gamma$. Moreover $\gamma = h \eta$ for $h:M \to \R$ where $\eta(x) \neq 0$ and $\eta(R)(x) =0$
for all $x$ in the support of $\gamma$. Given an almost complex structure adapted to $\eta$ we have that (the pull-back of) $\gamma\in L_2\Omega^1(X)$
and $\gamma\in \overline{dL_2(\mathcal{C}^\infty(X))}$ only if $\gamma=0$.
If $\gamma$ is exact then it extends to a $1$ form on $Y$ and the conclusions hold for the extension.
\end{itemize}

\end{theorem}

\begin{rem}\label{exact-gamma}
Before beginning the proof we give an example of such exact $\gamma$ on contact manifolds. Suppose the contact form has a closed Reeb
orbit and we can choose coordinates $(x_i,y_i,z) \in \R^{2(n-1)} \times \R \slash \Z$ in a tubular neighborhood of the orbit such that
the contact form is given by $\alpha = dz - \sum_{i=1}^{n-1}x_i dy_i$. Hence the Reeb vector field $R = \frac{\partial}{\partial z}$.
Let $\chi: [0,\infty) \to [0,1]$
satisfy $\chi(r)=1$ for $r$ close to $0$ but $\chi(r)=0$ for $r$ away from $0$.
Given $\chi$ we can define $\sigma$ similarly but so
that $\sigma'=-1$ whenever $\chi' \neq 0$. Then $\gamma = d(\chi(\sum_{i=1}^{n-1} x_i^2 + y_i^2))$ is an exact $1$ form of the form
$h \eta$ where $\eta =  d(\sigma(\sum_{i=1}^{n-1}x_i^2 + y_i^2)) \neq 0$ on the support of $\gamma$ and $\eta(R)=0$.
\end{rem}

For the proof of Theorem \ref{almost-kaehler}, we will need the following general
\begin{lemma}\label{dphi}
Let $(Z,g)$ be a Riemannian manifold and let $\gamma\in L_2\Omega^k(Z)$, $\gamma\neq 0$. Let
$\{d\varphi_j\}_{j\in\N}$ be a sequence in $d(L_2\Omega^{k-1}(Z))\cap L_2\Omega^k(X)$ such that
$d\varphi_j\to \gamma$ in $L_2$. Then, for $j>>1$,
$$
\lVert \varphi_j\lVert_{L_2(Z)}\geq C(\gamma)\,,
$$
for a suitable positive constant $C(\gamma)$.
\end{lemma}
\begin{proof} Since $\gamma\neq 0$ there exists a bump function $a$ such that $\langle \gamma,a\gamma\rangle>0$. We have:
$$
\begin{array}{lll}
\langle\gamma,a\gamma\rangle-\langle d\varphi_j, a\gamma\rangle&=&\langle\gamma-d\varphi_j,
a\gamma\rangle=
\int_Xg(\gamma-d\varphi_j,a\gamma)dx\\[10pt]
&\leq&\int_X\vert\gamma-d\varphi_j\vert \vert a\gamma\vert dx
\leq \lVert\gamma-d\varphi_j\lVert_{L_2(X)}\, \lVert a\gamma\lVert_{L_2(X)}.
\end{array}
$$
Set
$$
C_j=\lVert\gamma-d\varphi_j\lVert_{L_2(X)}\, \lVert a\gamma\lVert_{L_2(X)}.
$$
Note that $C_j\to 0$ for $j\to +\infty$.
We obtain
$$
\langle\gamma,a\gamma\rangle -C_j\leq\langle d\varphi_j, a\gamma\rangle=
\langle\varphi_j,\delta(a\gamma)\rangle\leq
\lVert\varphi_j\lVert_{L_2(X)}\, \lVert \delta(a\gamma)\lVert_{L_2(X)}\,.
$$
For $j$ large the left hand side is positive, hence $\lVert\delta(a\gamma)\lVert_{L_2(X)} >0$ and therefore setting
$$
C(\gamma)=\frac{\langle \gamma,a\gamma\rangle}{2\lVert\delta(a\gamma)\lVert_{L_2(X)}}
$$
we get
$$
\lVert\varphi_j\lVert_{L_2(X)}\geq C(\gamma)>0.
$$
\end{proof}
We give now the proof of Theorem \ref{almost-kaehler}
\begin{proofof}{Theorem \ref{almost-kaehler}}
By Remark \ref{extension} it suffices to work on $X$.
By construction, $J_{\varepsilon,f}$ is an almost complex structure on $X$ which is compatible with $\omega_\rho$. Therefore,
$g_{\rho,\varepsilon,f}(\cdot,\cdot)=\omega_\rho(\cdot,J_{\varepsilon,f}\cdot)$ is a Riemannian metric on $X$ and
$(X,\omega_\rho,J_{\varepsilon,f}, g_{\varepsilon,f,\rho})$ is an almost K\"ahler manifold. Then
$$
\omega_{\rho}^n=2\rho'\rho^{n-1}dt\wedge\alpha\wedge(d\alpha)^{n-1}
$$
is a volume form on $X$ and $\hbox{\rm Vol}_M=2\alpha\wedge(d\alpha)^{n-1}$ is a volume form on the compact
contact manifold $M$, so that
$$
\omega_{\rho}^n=\rho'\rho^{n-1}dt\wedge\hbox{\rm Vol}_M\,.
$$ \vskip.1truecm\noindent
i)
By assumption, $\omega=d(\rho\alpha)=d\lambda$, where $\lambda = \rho\alpha$; by definition $\omega_\rho$ is $d$-bounded if $\lambda\in L_\infty(X)$. 
Recalling that $J_{\varepsilon,f}R=-\frac{1}{f}\frac{d}{dt}$, we have,
$$
\vert R\vert^2=\omega_\rho(R,J_{\varepsilon,f}R)=-\omega_\rho(R,\frac{1}{f}\frac{d}{dt})=\frac{\rho'}{f}.
$$
Since $J_{\varepsilon,f}$ preserves the contact distribution $\xi$, we see that $\alpha$ is dual to $\frac{f}{\rho'}R$ with respect to $g_{\varepsilon,f,\rho}$.
Therefore $\vert\lambda\vert^2=\rho^2\frac{f}{\rho'}$. Hence $\lambda\in L_\infty(X)$ if and only if
\begin{equation}\label{d-bounded}
f\leq C\frac{\rho'}{\rho^2},
\end{equation}
where $C$ is a positive constant. By our assumptions,
$$\rho=\log t, \quad f=\frac{1}{t\log^2t},
$$
so that \eqref{d-bounded} is satisfied.\vskip.1truecm\noindent
ii) In order to check completeness of $(Y,\omega_\rho,J_{\varepsilon,f})$ it is enough to estimate
$\int_3^{+\infty}\vert \frac{d}{dt}\vert dt$.
We obtain:
$$
\vert \frac{d}{dt}\vert^2 =\omega_\rho(\frac{d}{dt},J_{\varepsilon,f}\frac{d}{dt})=
\omega_\rho(\frac{d}{dt},fR)=f\rho'.
$$
Therefore,
$$
\int_3^{+\infty}\vert \frac{d}{dt}\vert=\int_3^{+\infty}\sqrt{\rho'f}dt=+\infty\,,
$$
that is $(Y,\omega_\rho,J_{\varepsilon,f},g_{\varepsilon,f,\rho})$ is complete.\vskip.1truecm\noindent
iii) First of all we check that $\gamma\in L_2\Omega^1(X)$. We have the pointwise estimate valid on the support of $\gamma$:
$$
\vert \gamma\vert^2 = h^2 \vert \eta\vert^2 =h^2 \frac{(\eta(V))^2}{\vert V\vert^2}\leq \lVert h \lVert^2 _{L_\infty} \frac{\varepsilon}{\rho}\sum_{i=1}^{2n-2}a_i^2\,.
$$

Therefore, since $\varepsilon=\rho t^{1-n}$, for suitable constants $c_1$, $c_2$ we get:
$$
\lVert \gamma\lVert_{L_2(X)}^2\leq c_1 \int_X\frac{\varepsilon}{\rho} \omega_\rho^n=
c_1 \int_Xt^{1-n}\rho^{n-1}\rho'dt\wedge\hbox{\rm Vol}_M=
c_2 \int_3^{+\infty} \frac{\log^{n-1}t}{t^n}dt  <+\infty.
$$
Let $\gamma\neq 0$. We show that $\gamma\notin \overline{dL_2(\mathcal{C}^\infty(X))}$.\newline
By contradicton: assume that
there exists a sequence $\{\varphi_j\}_{j\in\N}$ in $L_2(\mathcal{C}^\infty(X))$ such that $d\varphi_j\in L^2\Omega^1(X)$ for every $j\in\N$ and
$d\varphi_j\to \gamma$ in
$L_2\Omega^1(X)$.

Set $\gamma_j=d\varphi_j$. 
We also write 
$$
f_j(t):=\int_{M\times\{t\}}(\varphi_j)^2 \,\hbox{\rm Vol}_M\,.
$$
Then
$$
\lVert\varphi_j\lVert^2_{L_2(M\times [a,b])}=\int_a^b\frac{(\log t)^{n-1}}{t}dt\int_{M\times\{t\}}(\varphi_j)^2 \,\hbox{\rm Vol}_M\, = \int_a^b\frac{(\log t)^{n-1}}{t}f_j(t)dt.
$$
We will show that $f_j(t)$ is bounded away from $0$ for large $j$, contradicting the assumption that $\varphi_j \in L_2(\mathcal{C}^\infty(X))$.

First, for the pointwise norm, since $\gamma(\frac{d}{dt})=0$, we have the estimate:
\begin{equation}\label{gamma}
\vert \gamma_j(\frac{d}{dt})\vert=\vert (\gamma_j-\gamma)(\frac{d}{dt})\vert
\leq\vert \gamma-\gamma_j\vert\,\vert\frac{d}{dt}\vert=\vert \gamma-\gamma_j\vert\sqrt{f\rho'}
=\frac{1}{t\log t}\vert \gamma-\gamma_j\vert.
\end{equation}

Now
$$
f_j'(t)=\int_{M\times\{t\}}2\varphi_jd\varphi_j(\frac{d}{dt}) \,\hbox{\rm Vol}_M=
\int_{M\times\{t\}}2\varphi_j\gamma_j(\frac{d}{dt}) \,\hbox{\rm Vol}_M\,.
$$
Therefore, by \eqref{gamma}, setting
$$
\psi_j(t):=2\lVert \gamma - \gamma_j\lVert_{L_2(M\times\{t\})},
$$
we obtain:
$$
\begin{array}{ll}
\vert f_j'(t)\vert&\leq \int_{M\times\{t\}}2\vert\varphi_j\vert\, \vert \gamma_j(\frac{d}{dt})\vert \,\hbox{\rm Vol}_M
\leq \frac{2}{t\log t}\lVert \varphi_j\lVert_{L_2(M\times\{t\})} \lVert \gamma - \gamma_j\lVert_{L_2(M\times\{t\})}\\[10pt]
&=\frac{1}{t\log t}\sqrt{f_j(t)}\psi_j(t).
\end{array}
$$
From the last expression,
\begin{equation}\label{psi}
\begin{array}{ll}
\Big[\sqrt{f_j(t)}\,\Big]^b_a&\leq\int_a^b\frac{1}{t\log t}\psi_j(t)dt
\leq\Big(\int_a^b\frac{1}{t (\log t)^{n+1}}dt\Big)^{\frac{1}{2}} \Big(\int_a^b \frac{(\log t)^{n-1}}{t} \psi^2_j(t)dt\Big)^{\frac{1}{2}}\\[10pt]
{}&=\sqrt{\Big[-\frac{1}{n(\log t)^{n}}\Big]^b_a}\Big(\int_a^b \frac{(\log t)^{n-1}}{t} \psi^2_j(t)dt\Big)^{\frac{1}{2}}.
\end{array}
\end{equation}
Now, by definition,
\begin{equation}\label{psigamma}
\int_a^b \frac{(\log t)^{n-1}}{t} \psi^2_j(t)dt = 
4\int_a^b\Big(\int_{M\times\{t\}}\vert\gamma-\gamma_j\vert^2\hbox{\rm Vol}_M\Big)\frac{(\log t)^{n-1}}{t}dt\leq
4\lVert\gamma-\gamma_j\lVert^2_{L_2(X)}.
\end{equation}
In view of \eqref{psi} and \eqref{psigamma}, we obtain
\begin{equation}\label{psigamma2}
\sqrt{f_j(b)}-\sqrt{f_j(a)}\leq 4\sqrt{\frac{1}{n(\log a)^{n}}-\frac{1}{n(\log b)^{n}}}\lVert \gamma-\gamma_j\lVert_{L_2(X)}.
\end{equation}
By assumption, $d\varphi_j\to\gamma$ in $L_2(X)$, and consequently $d\varphi_j\to\gamma$ also in $L_2(M\times[3,3+\delta])$.
By Lemma \ref{dphi}, it follows that there exists a constant $C(\gamma,\delta)$ such that
that
$$
\lVert\varphi_j\lVert^2_{L_2(M\times[3,3+\delta])}\geq C(\gamma,\delta)\,,
$$
for $j>>1$, where $C(\gamma,\delta) >0$ is independent of $j$, that is
$$
C(\gamma,\delta)\leq \int_3^{3+\delta}\frac{(\log t)^{n-1}}{t}f_j(t)dt.
$$
But
$$
\int_3^{3+\delta}\frac{(\log t)^{n-1}}{t}f_j(t)dt\leq \sup_{3\leq t\leq 3+\delta}\vert f_j(t)\vert \int_3^{3+\delta}\frac{(\log t)^{n-1}}{t}dt,
$$
which implies
$$
\sup_{3\leq t\leq 3+\delta}\vert f_j(t)\vert\geq \frac{n C(\gamma,\delta)}{(\log(3+\delta))^n-(\log 3)^n}.
$$
Therefore, for some $t\in[3,3+\epsilon]$, we have
$$
f_j(t)\geq  \frac{n C(\gamma,\delta)}{(\log(3+\delta))^n-(\log 3)^n}
$$
for large $j$, and we note that the lower bound is independent of $j$. By \eqref{psigamma2}, this implies that $f_j(t)$ is bounded below for large $j$ and all $t$, since $\gamma_j\to\gamma$ in $L_2(X)$. This gives our
contradiction as required. 
\vskip.1truecm\noindent
\end{proofof}
\begin{cor}\label{non-zero-cohomology}
Let $(Y,\omega_\rho,J_{\varepsilon,f}, g_{\varepsilon,f,\rho})$ be an almost K\"ahler structure adapted to $\eta\vert_{\hbox{supp}\,\gamma}$. For
$\eta$, $\gamma$, $\varepsilon$, $f$ as in iii) of Theorem \ref{almost-kaehler}, then
$$
L_2H^1(Y;\R)\neq\{0\}.
$$
\end{cor}
\begin{proof}
Let $\gamma$ be a non-zero exact $1$-form on $M$ satisfyng the hypothesis iii) of Theorem \ref{almost-kaehler}. Then the pull-back of $\gamma$ to
$X$ extends to a $1$-form on $Y$, still denoted by $\gamma$, such that $[\gamma]\neq 0$.
\end{proof}
\begin{cor}
Suppose $h_1,h_2,\ldots$ is an infinite family of linearly indipendent real valued smooth functions on $M$ such that
$\hbox{\rm supp}\,h_j\subset K$ and $\gamma_j=h_j\eta$ is exact for all $j$. Let $\gamma_j$ be an extension to
$Y$ of the pull-back of $\gamma_j$ to $X$. Then $(Y,\omega_\rho,J_{\varepsilon,f}, g_{\varepsilon,f,\rho})$ is an almost K\"ahler structure adapted
to $\eta\vert_K$ for $\eta$, $\gamma$, $f$ as in iii) of Theorem \ref{almost-kaehler}, and $L_2H^1(Y)$ is infinite dimensional.
\end{cor}
\begin{proof}
Any finite linear combination $\sum_{j=1}^r c_j\gamma_j$, for $c_i\in\R, j=1,\dots, r$ satisfies the assumptions of Corollary \ref{non-zero-cohomology}.
Hence
$$
\sum_{j=1}^r c_j[\gamma_j]\neq 0.
$$
\end{proof}
\begin{rem}
Such a family $\{\gamma_j\}$ can be constructed exactly as in Remark \ref{exact-gamma} by choosing the functions $\chi_j$ linearly independent.
\end{rem}
\begin{cor}
The almost complex structure $J_{\varepsilon,f}$ is not integrable.
\end{cor}
\begin{proof}
By \cite[1.2.B]{G}, complete $d$-bounded K\"ahler manifolds have $L_2H^1=\{0\}$.

\end{proof}
\begin{ex}
Let $\mathfrak{g}$ be the $(2n-1)$-dimensional Lie algebra whose dual basis $\mathfrak{g}^*$ has a basis
$\{V^1,\ldots,V^{2n-1}\}$ satisfying the following Maurer-Cartan equation:
$$
dV^{i}=0,\quad\hbox{\em for}\,\, i=1,\ldots ,2n-2,\qquad dV^{2n-1}=\sum_{i=1}^{n-1}V^{2i-1}\wedge V^{2i}.
$$
Accordingly, the dual basis $\{V_1,\ldots, V_{2n-1}\}$ satisfies the following commutation rules:
$$
[V_{2i-1},V_{2i}]=-V_{2n-1}
$$
and the other brackets vanish. Then the connected and simply-connected Lie group $G$ whose Lie algebra is $\mathfrak{g}$ has a lattice $\Gamma$ so that
$M=\Gamma\backslash G$ is a compact $(2n-1)$-dimensional manifold. The $1$-form $\alpha=V^{2n-1}$ gives rise to a contact structure on $M$.
Then $\gamma:=V^1$, $R=V_{2n-1}$ and the global frame $\{V_1,\ldots,V_{2n-1}\}$ satisfy the assumptions. Therefore, Theorem \ref{almost-kaehler}
applies to the manifold
$M\times (3,+\infty)$ with this $\gamma =\eta$.
\end{ex}
\section{$L_2$-Decomposition for almost complex $4$-manifolds}
Let $(X,J,g)$ be a $4$-dimensional almost Hermitian manifold. Then $J$ acts as an involution on the space of smooth $2$-forms
$\Omega^2(X)$: given $\alpha\in \Omega^2(X)$, for every pair of vector fields $u$, $v$ on $X$
$$
J\alpha (u,v)=\alpha(Ju,Jv)\,.
$$
Therefore the bundle $\Lambda^2X$ splits as the direct sum of $\pm 1$-eigenspaces $\Lambda_\pm$, i.e.,
$\Lambda^2X=\Lambda_+\oplus \Lambda_-$. We will refer to the sections of $\Lambda^+_J$, respectively $\Lambda^-_J$
as to the {\em invariant} respectively
{\em anti-invariant} forms, denoted by $\Omega^+(X)$, respectively $\Omega^-(X)$. Let  us denote by
$L_2\mathcal{Z}(X)$ the space of closed $2$-forms which are
in $L_2$ and set
$$
L_2\mathcal{Z}^\pm=L_2\mathcal{Z}(X)\cap \Omega^\pm(X).
$$
Define
$$
L_2H^\pm(X)=\{\mathfrak{a}\in L_2H^2(X;\R)\,\,\,\vert\,\,\, \exists\alpha\in L_2\mathcal{Z}^\pm\,\hbox{\rm such that}\,
\mathfrak{a}=[\alpha]\}.
$$
We will assume that $g$ is a complete $J$-Hermitian metric on $X$ and we will denote by $\omega$ the corresponding fundamental form. Let
$\Lambda^\pm_g$ be the $\pm 1$-eigenbundle of the $*$ Hodge operator associated with $g$ and $\omega$. Then,
we have the following relations
\begin{equation}\label{usefuls}
\Lambda^+_J=\Span_\R\langle\omega\rangle\oplus\Lambda^-_g, \qquad
\Lambda^+_g=\Span_\R\langle\omega\rangle\oplus\Lambda^-_J.
\end{equation}

In general if $\alpha^-$ is anti-invariant then 
\begin{equation}\label{antiinv}
*\alpha^-=\frac{1}{(n-2)!}\alpha^- \wedge \omega^{n-2}.
\end{equation}

We list some immediate consequences of these formulas on $4$-manifolds as a corollary.

\begin{cor}\label{useful} Closed anti-invariant forms are harmonic, that is, we have an inclusion $L_2\mathcal{Z}^-\hookrightarrow \mathcal{H}^2_2$.
All anti-invariant forms are self-dual, while anti self-dual forms are invariant.
\end{cor}

For closed almost complex manifolds Dr\v{a}ghici, Li and  Zhang showed in \cite{DLZ} that there is a direct sum decomposition
$$
H^2_{dR}(X;\R)=H^+(X)\oplus H^-(X).
$$
In this section we generalize such a decomposition to the $L_2$ setting. The arguments follow closely those in \cite{DLZ}.\newline
First of all, by the $L_2$ Hodge decomposition and Lemma \ref{d-closed} the vector space $L_2H^2(X;\R)$ is isomorphic to the space $\mathcal{H}^2_2$ of $L_2$-harmonic forms on $X$,
which is a topological
subspace of of the Hilbert space $L_2\Omega^2(X)$. The following lemma is well known.
\begin{lemma}
$\mathcal{H}^2_2$ is a closed subspace of $L_2\Omega^2(X)$, and hence inherits the structure of a Hilbert space.
\end{lemma}
\begin{proof}
We recall the proof for the sake of completeness. Let $\{\alpha_j\}_{j\in\N}$ be a sequence in $\mathcal{H}^2_2$ such that
$\alpha_j\to \alpha$, for $j\to +\infty$ in $L_2$. Then, for every smooth compactly supported $2$-form $\varphi$ on $X$ we have:
$$
\langle\alpha,\delta\varphi\rangle =\lim_{j\to+\infty}\langle \alpha_j,\delta\varphi\rangle =0.
$$
In the same way,
$$
\langle\alpha,d\varphi\rangle =\lim_{j\to+\infty}\langle \alpha_j,\varphi\rangle =0,
$$
that is $\alpha$ is harmonic in the sense of distributions. Therefore, by elliptic regularity, $\alpha\in \mathcal{H}^2_2$.
\end{proof}
\begin{lemma}\label{self-dual}
Let $\alpha\in L_2\Omega^2(X)$ be self-dual and let $\alpha=\alpha_H+\lambda+\mu$ be its $L_2$-Hodge decomposition \eqref{L2Hodge}.
Then,
$$
\lambda^{sd}_g=\mu^{sd}_g \,,\qquad \lambda^{asd}_g=-\mu^{asd}_g
$$
where $\lambda=\lambda^{sd}_g+\lambda^{asd}_g$ and $\mu=\mu^{sd}_g+\mu^{asd}_g$denote the $*$-decomposition. Furthermore,
the forms
\begin{equation}\label{closed-self-dual}
\alpha+2\lambda^{asd}_g=\alpha_H +2\lambda.
\end{equation}
are closed.
\end{lemma}
\begin{proof}
By assumption $*\alpha=\alpha$. Hence, if $\alpha=\alpha_H+\lambda+\mu$, where $\lambda\in\overline{d(L_2\Omega^1(X))}$,
$\mu\in\overline{\delta(L_2\Omega^3(X))}$ then,
$$
*\alpha =*\alpha_H+*\lambda+*\mu =\alpha_H+\lambda+\mu
$$
Now, if $\{\lambda_j\}_{j\in\N}$, $\{\mu_j\}_{j\in\N}$ are sequences in $d(L_2\Omega^1(X))$, respectively
$\delta(L_2\Omega^3(X))$ such that
$$
\lambda_j=d\lambda'_j,\quad \mu_j=d\mu'_j, \quad \lambda_j\in L_2\Omega^1(X), \mu_j\in L_2\Omega^3(X), \quad
d\lambda'_j\to \lambda, d\mu'_j\to \mu
$$
in the $L_2$-norm, then
$$
\lVert\lambda_j-\lambda\lVert_{L_2(X)}=\lVert d \lambda'_j-\lambda\lVert_{L_2(X)}=
\lVert *d\lambda'_j-*\lambda\lVert_{L_2(X)},
$$
so that $*d\lambda'_j\to *\lambda$ in $L_2$ and, similarly, $*d\mu'_j\to *\mu$. Therefore, since
$$
*d\lambda'_j\in \delta(L_2\Omega^{3}(X),\quad *\delta\mu'_j\in d(L_2\Omega^{2}(X),
$$
we obtain that
$$
*\lambda\in L_2\Omega^2(X), \quad*\mu\in L_2\Omega^2(X),\quad
*\lambda \in \overline{\delta(L_2\Omega^{3}(X)},\quad *\mu\in\overline{d(L_2\Omega^{2}(X)}.
$$
Therefore, by the uniqueness of the $L_2$-Hodge decomposition,
$$
*\lambda=\mu,\qquad *\mu=-\lambda.
$$
Then \eqref{closed-self-dual} follows.
\end{proof}
\begin{lemma}
The following holds
$$
\overline{L_2H^+(X)}\cap \overline{L_2H^-(X)}=\{0\}.
$$
\end{lemma}
\begin{proof}
Let $\{\alpha_i\}_{i\in\N}$, $\{\beta_j\}_{i\in\N}$  be sequences of harmonic forms in $L_2\Omega^2(X)$
with $[\alpha_i] \in L_2 H^+(X)$ and $[\beta_i] \in L_2 H^-(X)$ such that
$\alpha_i\to \alpha$, for $i\to +\infty$ in $L_2$ and $\beta_i\to \alpha$, for $i\to +\infty$ in $L_2$. Then,
using Lemmas \ref{d-closed} and Corollary \ref{useful} we can write
$$
\alpha_i=\theta^+_i+\lambda_i\,,\qquad \beta_i=\eta^-_i
$$
where $\theta^+_i\in L_2\mathcal{Z}^+$, $\eta^-_i\in L_2\mathcal{Z}^-$, $\lambda_i\in\overline{dL_2(\Omega^1(X))}$. Then, as anti-invariant
forms are self-dual we have
$$
0=\int_X\theta^+_i\wedge\eta^-_i=\int_X\theta^+_i\wedge *\eta^-_i=\int_X\alpha_i\wedge *\beta_i=\langle \alpha_i,\beta_i\rangle.
$$
Taking a limit this implies
$$
\lVert \alpha\lVert^2_{L_2(X)}=0,
$$
that is, $\overline{L_2H^+(X)}\cap \overline{L_2H^-(X)}=\{0\}$.
\end{proof}
\begin{lemma}\label{noint}
$$
(L_2H^+(X)\oplus L_2H^-(X))^{\perp}=\{0\}.
$$
\end{lemma}
The orthogonal complement is defined by recalling that $L_2 H^{\pm}(X)$ can be thought of as subspaces of the Hilbert space $\mathcal{H}^2_2$.

\begin{proof}
By contradiction: assume that there exists $[\alpha]\in L_2H^2(X;\R)$ such that, for every $[\theta^+]+[\theta^-]\in L_2H^+(X)\oplus L_2H^-(X)$,
$$
<\alpha,\theta^+ +\theta^->=0.
$$
To compute the inner product we assume that $\alpha$, $\theta^+$ and $\theta^-$ are harmonic representatives.

By taking $\theta^+$ to be anti self-dual part of $\alpha$ (which is invariant by Corollary \ref{useful})
and $\theta^-=0$ we see immediately that the anti self-dual part of $\alpha$ must vanish, that is, $\alpha$ is self-dual.


Therefore, by \eqref{usefuls} we have
$$
\alpha =c\omega + \theta^-,
$$
where $c$ is a function on $X$ such that $c\neq 0$ and $\theta^-\in \Omega^-(X)$.
\newline
We may assume that there exists $x\in X$ such that $c(x) >0$. Let $a$ be a bump function and $W$ be a compact neighborhood of $x$ such that
$a\vert_W =1$ and
$$
\hbox{\rm supp}\,a\subset\{x\in X\,\,\,\vert \,\,\, c(x)>0\}.
$$
Let $\Phi:X\to \R$ be defined as
$$
\Phi(x)=g(\alpha, a\omega)(x).
$$
Then
$$
\Phi(x)=g(\alpha, a\omega)(x)=g(c\omega +\theta^-,a\omega)(x)=c(x)a(x)\geq 0.
$$
Now we apply Lemma \ref{self-dual} to the self-dual form $a\Phi\omega$. Let $\lambda$ be the exact part of the $L_2$ Hodge decomposition of $a\Phi\omega$.
Then Lemma \ref{self-dual} gives
$$
(a\Phi\omega)_H+2 \lambda=a\Phi\omega + 2\lambda^{asd}_g \in L_2H^+(X).
$$
Therefore, using Lemma \ref{stokes} and noting that self-dual and anti self-dual forms are pointwise $g$-orthogonal, we obtain
$$
\begin{array}{ll}
0 &=<\alpha, (a\Phi\omega + 2\lambda^{asd}_g)_H> = \int_X \alpha \wedge(a \Phi \omega+2\lambda^{asd}_g)
\\[10pt]
&=\int_X \alpha \wedge *(a \Phi \omega+2\lambda^{asd}_g)=\int_Xg(\alpha,a\Phi\omega)-2g(\alpha,\lambda^{asd}_g)\hbox{\rm Vol}_X\\[10pt]
&=\int_X g(\alpha,a \Phi \omega)\hbox{\rm Vol}_X=\int_X\Phi^2\hbox{\rm Vol}_X.
\end{array}
$$
Hence $\Phi=0$ and $c(x)=0$. This gives a contradiction.
\end{proof}
\begin{lemma}\label{fulls} We have
$$
L_2H^2(X;\R)=\overline{L_2H^+(X)\oplus L_2H^-(X)}.
$$
\end{lemma}
\begin{proof}
$$\overline{L_2H^+(X)\oplus L_2H^-(X)} = ((L_2H^+(X)\oplus L_2H^-(X))^{\perp})^{\perp} = \{0\}^{\perp} = \mathcal{H}^2_2$$
using Lemma \ref{noint}.
\end{proof}
\begin{lemma}\label{closedsum}
The subspace $\overline{L_2H^+(X)}\oplus\overline{L_2H^-(X)}$ is closed in $L_2H^2(X;\R)$.
\end{lemma}
\begin{proof}
As $\overline{L_2 H^+(X)}$ and $\overline{L_2 H^-(X)}$ are orthogonal, we can check that a sequence $\{(\alpha_i, \beta_i)\}$ in the direct sum is Cauchy if and only if both $\{\alpha_i\}$ and $\{\beta_i\}$ are Cauchy.
\end{proof}
\begin{theorem} \label{decomp}
Let $(X,J,g)$ be an complete almost Hermitian $4$-dimensional manifold.
Then, we have the following decomposition
$$
L_2H^2(X;\R)=\overline{L_2H^+(X)}\oplus\overline{L_2H^-(X)}.
$$
\end{theorem}
\begin{proof}Indeed, by Lemma \ref{closedsum} the direct sum is closed and so by Lemma \ref{fulls} contains $L_2H^2(X;\R)=\overline{L_2H^+(X)\oplus L_2H^-(X)}$.

\end{proof}

Let now $J$ be an almost complex structure on a manifold $X$ of any dimension. The following Proposition provides a cohomological obstruction on $J$ in order that there
exists a compatible symplectic form $\omega$ such that the associated Hermitian metric $g_J(\cdot,\cdot)=\omega(\cdot,J\cdot)$ is complete.
\begin{prop}
Let $(X,J,\omega,g_J)$ be an almost K\"ahler manifold such that $g_J$ is complete. Then
$$
L_2H^+(X)\cap L_2H^-(X)=\{0\}.
$$
\end{prop}
\begin{proof}
Let $[\alpha]\in L_2H^+(X)\cap L_2H^-(X)$. Then, there exist $\alpha^\pm\in L_2\mathcal{Z}^\pm$ such that
$$
\alpha = \alpha^+ +\lambda\,,\qquad \alpha=\alpha^- +\mu,
$$
where $\alpha^\pm\in L_2\mathcal{Z}^\pm$. Then
\begin{equation}\label{decomposition-n}
\alpha^+=\alpha^-+\eta
\end{equation}
where $\eta\in\overline{d(L_2A^1(X))}$.
Let $\{\eta_j\}_{j\in\N}$ be a sequence in $L_2(X)$ such that
$\eta_j=d\eta'_j$, $\eta'_j\in L_2(X)$ and $d\eta'_j\to \eta$ in $L_2(X)$. Then, by bi-degree reasons,
\begin{equation}\label{invariant-anti-invariant}
[\alpha^+]\cup [\alpha^-]\cup[\omega^{n-2}]=\int_X\alpha^+\wedge\alpha^-\wedge\omega^{n-2} =0.
\end{equation}
We claim that
\begin{equation}\label{stokes-anti-invariant-n}
\int_X\eta\wedge\alpha^-\wedge\omega^{n-2} =0.
\end{equation}
Indeed,
$$
\begin{array}{ll}
\vert\langle d\eta'_j\wedge\alpha^-,*\omega^{n-2}\rangle -\langle \eta\wedge\alpha^-,*\omega^{n-2}\rangle \vert&\leq
\int_X \vert (d\eta'_j-\eta)\wedge \alpha^-\vert \vert *\omega^{n-2}\vert \hbox{\rm Vol}_X\\[10pt]
&\leq C\int_X \vert (d\eta'_j-\eta)\vert \vert\alpha^-\vert \hbox{\rm Vol}_X\\[10pt]
&\leq\lVert d\eta'_j-\eta\lVert_{L_2(X)}
\lVert \alpha^- \lVert_{L_2(X)},
\end{array}
$$
that is $\langle d\eta'_j\wedge\alpha^-,*\omega^{n-2}\rangle\to\langle \eta\wedge\alpha^-,*\omega^{n-2}\rangle$, for $j\to+\infty$.
On the other hand, by Lemma \ref{stokes},
$$
0= \lim_{j \to \infty} \langle d\eta'_j\wedge\alpha^-,*\omega^{n-2}\rangle =\int_X \eta\wedge\alpha^-\wedge\omega^{n-2},
$$
that is, \eqref{stokes-anti-invariant-n}. Therefore, by \eqref{antiinv}, \eqref{decomposition-n} and \eqref{invariant-anti-invariant} we have
$$
\begin{array}{ll}
[\alpha^+]\cup[\alpha^-]\cup[\omega^{n-2}]&=\int_X(\alpha^-+\eta)\wedge\alpha^-\wedge\omega^{n-2}
=\int_X\alpha^-\wedge\alpha^-\wedge\omega^{n-2}\\[10pt]
&=(n-2)!\int_X\alpha^-\wedge *\alpha^-=\lVert\alpha^-\lVert_{L_2(X)}.
\end{array}
$$
Hence $[\alpha]=0$.
\end{proof}
\begin{ex}
Let ${\Delta}^2=\{(z_1,z_2)\in\C^2\,\,\,\vert\,\,\, \vert z_1\vert<r_1,\,\vert z_2\vert<r_2\}$ be a polydisc in $\C^2$ endowed with the complete and
$d$-bounded K\"ahler metric
$$
\omega=i\partial\overline{\partial}(\sum_{j=1}^2\log(1-\vert z_j\vert^2).
$$
Then, the real $J$-anti-invariant forms
$$
\frac{1}{2}(dz_1\wedge dz_2+d\overline{z}_1\wedge d\overline{z}_2)\,,\qquad
\frac{1}{2i}(dz_1\wedge dz_2-d\overline{z}_1\wedge d\overline{z}_2)
$$
and the real $J$-invariant forms
$$
\frac{1}{2}(dz_1\wedge d\overline{z}_2+d\overline{z}_1\wedge d z_2)\,,\qquad
\frac{1}{2i}(dz_1\wedge d\overline{z}_2-d\overline{z}_1\wedge dz_2)
$$
are $L_2$-harmonic, so that $L_2H^\pm(\Delta^2)\neq\{0\}$.
\end{ex}
\begin{rem}
Notice that for the de Rham cohomology, Dr\v{a}ghici, Li and Zhang in \cite[Theorem 3.24]{DLZ1} constructed non-compact complex 
surfaces for which $H^+(M)\cap H^-(M)\neq\{0\}$. 
\end{rem}

\end{document}